\documentclass[10pt,reqno]{amsart}
\usepackage[colorlinks=true,allcolors=blue
]{hyperref}
\usepackage{color}
\usepackage{amsmath, amssymb, amsthm}

\usepackage{mathrsfs}
\usepackage{mathtools}
\usepackage[noabbrev,capitalize,nameinlink]{cleveref}
\crefname{equation}{}{}
\usepackage{fullpage}
\usepackage[noadjust]{cite}
\usepackage{graphics}
\usepackage{pifont}
\usepackage{tikz}
\usepackage{bbm}
\usepackage[T1]{fontenc}

\usetikzlibrary{arrows.meta}

\usepackage{environ}
\usepackage{framed}
\usepackage{url}
\usepackage[linesnumbered,ruled,vlined]{algorithm2e}
\usepackage[noend]{algpseudocode}
\usepackage[labelfont=bf]{caption}
\usepackage{cite}
\usepackage{framed}
\usepackage[framemethod=tikz]{mdframed}
\usepackage{appendix}
\usepackage{graphicx}
\usepackage[textsize=tiny]{todonotes}
\usepackage{tcolorbox}
\usepackage{enumerate}
\usepackage[shortlabels]{enumitem}
\allowdisplaybreaks[1]

\apptocmd{\sloppy}{\hbadness 10000\relax}{}{} % magic BibTeX spacing fix

\crefname{algocf}{Algorithm}{Algorithms}

\crefname{equation}{}{} %remove ``Equation''
\crefname{conjecture}{Conjecture}{Conjectures} %add ``Conjecture''
\AtBeginEnvironment{appendices}{\crefalias{section}{appendix}} %appendices

\let\originalleft\left
\let\originalright\right
\renewcommand{\left}{\mathopen{}\mathclose\bgroup\originalleft}
\renewcommand{\right}{\aftergroup\egroup\originalright}

\crefformat{enumi}{#2#1#3}
\crefrangeformat{enumi}{#3#1#4 to~#5#2#6}
\crefmultiformat{enumi}{#2#1#3}%
{ and~#2#1#3}{, #2#1#3}{ and~#2#1#3}

\usepackage[color,final]{showkeys} %add in 'final' into parameter to remove showkeys

\colorlet{refkey}{orange!20}
\colorlet{labelkey}{blue!30}

\crefname{algocf}{Algorithm}{Algorithms}

\usepackage{aliascnt}

\numberwithin{equation}{section}
\newtheorem{theorem}{Theorem}[section]
\crefname{theorem}{Theorem}{Theorems}
\Crefname{theorem}{Theorem}{Theorems}

\makeatletter
\newcommand{\newthmwithalias}[3]{%
  \newaliascnt{#1}{theorem}%
  \newtheorem{#1}[#1]{#2}%
  \aliascntresetthe{#1}%
  \expandafter\crefname\expandafter{#1}{#2}{#2s}%
  \expandafter\Crefname\expandafter{#1}{#2}{#2s}%
  \expandafter\def\csname #1autorefname\endcsname{#2}%
}
\makeatother

\newthmwithalias{proposition}{Proposition}{}
\newthmwithalias{lemma}{Lemma}{}
\newthmwithalias{claim}{Claim}{}
\newthmwithalias{corollary}{Corollary}{}
\crefname{corollary}{Corollary}{Corollaries}
\newthmwithalias{conjecture}{Conjecture}{}
\newthmwithalias{fact}{Fact}{}
\newthmwithalias{definition}{Definition}{}
\newthmwithalias{problem}{Problem}{}
\newthmwithalias{question}{Question}{}
\newthmwithalias{assumption}{Assumption}{}
\newthmwithalias{example}{Example}{}
\newthmwithalias{setup}{Setup}{}

\theoremstyle{remark}
\newtheorem*{remark}{Remark}

\newtheorem*{question*}{Question}
\newtheorem*{definition*}{Definition}

\crefname{section}{Section}{Sections}
\Crefname{section}{Section}{Sections}
\crefname{subsection}{Section}{Sections}
\Crefname{subsection}{Section}{Sections}
\crefname{equation}{}{} % keep equations as just numbers
\crefname{algocf}{Algorithm}{Algorithms}
\Crefname{algocf}{Algorithm}{Algorithms}

\newcommand{\mb}{\mathbb}

\newcommand{\mc}{\mathcal}

\newcommand{\on}{\operatorname}

\newcommand{\eps}{\varepsilon}

\newcommand{\hide}[1]{}

\allowdisplaybreaks

\title{Permanents of random matrices over finite fields}
\author[A1]{Zach Hunter}
\address{Department of Mathematics, ETH Z\"{u}rich, Z\"{u}rich, Switzerland.}
\email{zach.hunter@math.ethz.ch}

\author[A2]{Matthew Kwan}
\address{Institute of Science and Technology Austria (ISTA). Am Campus 1, 3400 Klosterneuburg, Austria}
\email{matthew.kwan@ist.ac.at}

\author[A3]{Lisa Sauermann}
\address{Institute for Applied Mathematics, University of Bonn, Germany}
\email{sauermann@iam.uni-bonn.de}

\thanks{
Zach Hunter was supported by SNSF grant 200021-228014. Matthew Kwan was supported by ERC Starting Grant ``RANDSTRUCT'' No.\ 101076777. Lisa Sauermann was supported by the Deutsche Forschungsgemeinschaft (DFG, German Research Foundation) -- CRC 1720 -- 539309657. 
}

\begin{document}

\begin{abstract} 
Fix a finite field $\mb F_q$ and let $A\in \mb F_q^{n\times n}$ be a uniformly random $n\times n$ matrix over $\mb F_q$. The asymptotic distribution of the determinant $\det(A)$ is well-understood, but the asymptotic distribution of the permanent $\on{per}(A)$ is still something of a mystery. In this paper we make a first step in this direction, proving that $\on{per}(A)$ is significantly more uniform than $\on{det}(A)$.
\end{abstract}

\maketitle

\section{Introduction}\label{sec:introduction}
Two important matrix parameters are the \emph{determinant}
and \emph{permanent}: for an $n\times n$ matrix $A=(a_{i,j})$, 
they are defined as
\begin{equation}
\det(A)=\sum_{\pi\in S_{n}}\operatorname{sign}(\pi)\prod_{i=1}^{n}a_{i,\pi(i)}\qquad\text{and}\qquad\operatorname{per}(A)=\sum_{\pi\in S_{n}}\prod_{i=1}^{n}a_{i,\pi(i)}.\label{eq:det-per}
\end{equation}
If we fix a finite field $\mb F_q$ and let $A\in \mb F_q^{n\times n}$ be a uniformly random $n\times n$ matrix over $\mb F_q$, it is easy to compute the limiting distribution of $\det(A)$ (this is essentially equivalent to computing the number of invertible matrices over $\mb F_q$, which goes all the way back to Galois \cite[p.\ 280]{Smith-book}).
\begin{fact}\label{fact:det}
For a prime power $q$, let 
\[
\alpha_{q}=1-\prod_{i=1}^{\infty}(1-q^{-i}).%=\frac{1}{q}+O\left(\frac{1}{q^{2}}\right).
\]
Then, for a uniformly random $n\times n$ matrix $A\in\mathbb{F}_{q}^{n\times n}$ and any $x\in \mb F_q$, we have
\[
\lim_{n\to \infty}\Pr[\det(A)=x]=\begin{cases}
\alpha_{q} & \text{if }x=0\\
(1-\alpha_{q})/(q-1) & \text{otherwise.}
\end{cases}
\]
\end{fact}
Notably, this limiting distribution is \emph{not} uniform over $\mb F_q$: the determinant is significantly more likely to take the value zero than any other value. For example, when $q=3$ we have $\alpha_q\approx 0.44$, which is significantly larger than $1/q=1/3$. Roughly speaking, this is because the event $\{\det(A)=0\}$ describes the rows of $A$ being linearly dependent: if we reveal $A$ row-by-row, and at some step we discover that a row lies in the span of the previous rows, this tells us that $\det(A)=0$ regardless of what happens in future rows. On the other hand, if after revealing the second-last row we still haven't discovered a linear dependence relation between the rows, then it is easy to see that the conditional distribution of $\det(A)$ is uniform over $\mb F_q$.

Unlike the determinant, the permanent does not seem to have any linear-algebraic meaning (unless our field has characteristic 2, in which case the determinant and permanent are equal). It is tempting to conjecture that, if $\mb F_q$ does not have characteristic 2, then the permanent of a uniformly random matrix over $\mb F_q$ is asymptotically uniform, as follows.
\begin{conjecture}
\label{conj:main}Fix a finite field $\mathbb{F}_{q}$ of odd characteristic.
Then, for a uniformly random $n\times n$ matrix $A\in\mathbb{F}_q^{n\times n}$,
and any $x\in\mathbb{F}_q$, we have
\[
\lim_{n\to\infty}\Pr[\on{per}(A)=x]=\frac{1}{q}.
\]
\end{conjecture}
This conjecture seems to have been floating around the community for a while, but we are not completely sure about its origins. Permanents of random matrices over finite fields have received quite some interest in the computer science community (see e.g.\ \cite{FL96,Lip91,GLRSW91,GS92,CPS99}), due to a phenomenon called \emph{random self-reducibility}, but surprisingly we were not able to find any study of the asymptotic distribution of $\on{per}(A)$ in this literature. The earliest reference we could find was in the open problem session of a 2017 workshop on combinatorics at the Mathematisches Forschungsinstitut Oberwolfach~\cite{Kop17}, in which Swastik Kopparty asked for the limiting distribution of the permanent over $\mb F_3$. The $\mb F_3$ case of \cref{conj:main} also appears explicitly in a recent paper by Scheinerman~\cite{Sch24} (backed by quite convincing computational evidence), and the general case of \cref{conj:main} appeared in a very recent paper of Ghasemi, Gross and Kopparty~\cite{GGK25}. See also a related conjecture by Esperet~\cite{Esp22} (motivated by questions in graph theory).

Note that, by symmetry (through rescaling any row), $\on{per}(A)$ is equally likely to take any nonzero value. Thus, in order to understand the distribution of $\on{per}(A)$ it suffices to understand $\Pr[\on{per}(A)=0]$. More specifically, the assertion of \cref{conj:main} is equivalent to $\lim_{n\to\infty}\Pr[\on{per}(A)=0]=1/q$.

Although we are not able to prove \cref{conj:main}, we are at least
able to prove that the permanent is (asymptotically) \emph{more} uniform than the determinant,
as follows (recall that $\lim_{n\to\infty}\Pr[\det(A)=0]=\alpha_q>1/q$ and observe that $\alpha_q-1/q$ has order of magnitude $1/q^2$).
\begin{theorem}
\label{thm:main}Fix a finite field $\mathbb{F}_{q}$ of odd characteristic. For a uniformly random $n\times n$ matrix $A\in\mathbb{F}_{q}^{n\times n}$,
we have
\begin{equation}
\Pr[\on{per}(A)=0]\ge\frac{1}{q}\label{eq:trivial-lower-bound}
\end{equation}
for all $n$, and we have the \textbf{strict} inequality
\begin{equation}
\limsup_{n\to\infty}\Pr[\on{per}(A)=0]<\alpha_{q}\label{eq:separation-all-p}.
\end{equation}
Moreover, for all $n\ge 3$ we have
\begin{equation}
\Pr[\on{per}(A)=0]\le \frac{1}{q}+\frac{C}{q^{3}}\label{eq:asymptotic-p}
\end{equation}
for some absolute constant $C$. %\mk{rephrased this slightly to try to avoid it sounding like the main result was the trivial 1/q lower bound, and to put the statements in the same order as theorem 1.4}
\end{theorem}

For $n=2$ one has $\Pr[\on{per}(A)=0]=1/q+\Theta(1/q^2)$, so the assumption $n\ge 3$ in \cref{eq:asymptotic-p} cannot be omitted.

We remark that a theorem in a similar spirit was previously proved by Budrevich and Guterman~\cite{BG12} (see also the simpler proof in \cite{Bass13}, and the earlier work \cite{DGKO11}): restated in probabilistic language, their main result was that for any finite field $\mathbb F_q$ of odd characteristic, and any $n\ge 3$, the zero-permanent probability is strictly less than the zero-determinant probability. However, the gap they established between these two probabilities tends to zero very rapidly as $n\to \infty$. Budrevich and Guterman were motivated by \emph{P\'olya's permanent problem} (see e.g.\ \cite{RST99,MM61,VY89}), which asks to what extent it is possible to re-express the permanent of a matrix in terms of a determinant. The gap established in \cite{BG12} shows that there is no bijective ``converter'' $f:\mb F_q^{n\times n}\to \mb F_q^{n\times n}$ such that $\on{per}(A)=\det(f(A))$ for all $A\in \mb F_q^{n\times n}$. Note that \cref{thm:main} implies the stronger result that there is no such converter that works for \emph{almost} all $A\in \mb F_q^{n\times n}$.

\subsection{General distributions}
It turns out that \cref{fact:det} is \emph{universal}, in the sense that the distribution of the determinant of a random matrix $A\in \mb F_q^{n\times n}$ does not depend very strongly on the distribution of the entries of $A$. For example, it is known that for \emph{any} non-constant distribution $\mu$ over a  finite field $\mb F_p$ of prime order, if $A\in \mb F_p^{n\times n}$ is a random matrix with independent $\mu$-distributed entries, then $\lim_{n\to \infty }\Pr[\det(A)=0]=\alpha_p$. This seems to have been first proved in 1990 by Charlap, Rees and Robbins~\cite{CRR90}, though similar theorems with more restrictive assumptions were proved much earlier by Kozlov~\cite{Koz} and Kovalenko and Levitskaja~\cite{KL75}, and a very general theorem over arbitrary finite fields (necessarily with a more technical statement) was later proved by Kahn and Koml\'os~\cite{KK01}. See also \cite{Map10,Map13,Woo17,Woo19,NW22,Ebe22,LMN21} for some more recent work on universality of various algebraic properties for random matrices over finite fields.

We can prove a theorem in the spirit of \cref{thm:main} for arbitrary (non-constant) distributions over finite fields of prime order, as follows.
\begin{theorem}\label{thm:general}
Fix a finite field $\mathbb{F}_{p}$ of prime order $p\ge 3$. Then, there is $\delta_p>0$ such that the following holds.
Let $\mu$ be a probability distribution over $\mb F_p$ supported on more than one value. If $A\in\mathbb{F}_{p}^{n\times n}$ is a random matrix with independent $\mu$-distributed entries, then for every $z\in \mb F_p$ we have
\begin{equation}
\limsup_{n\to\infty}\Pr[\on{per}(A)=z]\le\alpha_{p}-\delta_{p} \label{eq:separation-general}   
\end{equation}
and
\begin{equation}
\frac{1}{p}-\frac{C}{p^{3}}\le \liminf_{n\to\infty}\Pr[\on{per}(A)=z]\le \limsup_{n\to\infty}\Pr[\on{per}(A)=z]\le \frac{1}{p}+\frac{C}{p^{3}}.    \label{eq:asymptotic-general}
\end{equation}
for some absolute constant $C$.
\end{theorem}
\begin{remark}
For simplicity we have stated \cref{thm:general} only for prime fields, but our proof approach surely can be adapted to arbitrary finite fields, under some kind of non-degeneracy assumption on the distributions of the entries. Unfortunately, our proof approach does not seem to be compatible with the weakest non-degeneracy conditions in the literature. For example, Kahn and Koml\'os~\cite{KK01} proved that as long as $\mu$ is a distribution over a finite field $\mb F_q$ that is not supported on an affine translate of a subfield of $\mb F_q$, then for a random matrix $A$ with independent $\mu$-distributed entries we have $\lim_{n\to\infty}\Pr[\det(A)=0]=\alpha_q$. In this very general setting we do not even know how to prove the inequality $\limsup_{n\to\infty}\Pr[\on{per}(A)=x]\le \alpha_q$.
\end{remark}

\begin{remark}
In combinatorial random matrix theory, perhaps the most intensively studied class of random matrices is the class of random \emph{sign matrices}. Letting $A\in \{-1,1\}^{n\times n}$ be a uniformly random $n\times n$ matrix with $\pm 1$ entries, in spectacular recent work, Tikhomirov~\cite{Tik20} proved that $\Pr[\det(A)=0]=(1/2+o(1))^n$ (see also the surveys \cite{Vu14,Vu20,Gui23,Sah25}). There has been quite some interest in permanents of random sign matrices (see e.g. \cite{TV09,LM19,KS22,HKS25}), but still much less is known about permanents than determinants. It seems quite plausible that $\on{per}(A)$ and $\det(A)$ have very similar behaviour, but it is worth noting that \cref{eq:asymptotic-general} gives one of the first ways in which the distributions of $\on{per}(A)$ and $\det(A)$ \emph{differ}: their statistics mod $p$ are quite different, for all primes $p\ge 3$.
\end{remark}

\subsection{Overview of the paper and proofs} It would be possible to give a unified proof of \cref{thm:main,thm:general}, but we found it to be clearer to present the proofs separately.

In \cref{sec:uniform} we give the short proof of \cref{thm:main}. We reveal the rows of our random matrix one-by-one and study the permanents of submatrices as we go, inductively showing that ``nonzero permanents are at least as likely as nonzero determinants''. Then, in the final few rows we proceed slightly differently, exploiting a certain symmetry property of the permanent.

In \cref{sec:general} we prove \cref{thm:general}. The overall strategy is the same, but we need to additionally exploit the fact that if we add a large number of independent $\mu$-distributed random variables, the result is nearly uniform. So, if we can simultaneously keep track of the permanents of many submatrices, we can ``simulate'' the uniform case in the $\mu$-distributed setting. This requires a few additional ideas.

\section{Proof for the uniform case}\label{sec:uniform}
In this section we prove \cref{thm:main}. First, it is easy to see that the permanent enjoys an analogue of the minor expansion formula for the determinant, as we record below. This fact will allow us to study permanents recursively.
\begin{fact}\label{fact:minor-expansion}
For any matrix $A\in \mb F_q^{k\times k}$, we have
\[
\on{per}(A)=\sum_{i=1}^k\on{per}(A'_i) x_{i},
\]
where $A'_i$ is the $(k-1)\times (k-1)$ submatrix of $A$ obtained by deleting the last row and the $i$th column, and $x_i$ is the entry of $A$ in the last row and the $i$th column.
\end{fact}

Now, the following lemma shows that nonzero permanent-minors ``grow
at least as well'' as nonzero determinant-minors, each time we reveal an additional row (note that for a uniformly random matrix $A\in \mb F_q^{n\times n}$, when conditioning on linearly independent outcomes of the first $n-s$ rows, the probability that the next row is in their span is precisely $q^{-s}$).
\begin{definition}
In the setting of \cref{thm:main}, for any $0\le s\le n$, let $A^{\uparrow s}$ be the $(n-s)\times n$ submatrix
of $A$ obtained by deleting the last $s$ rows. Let $\mathcal{E}(s)$
be the event that $A^{\uparrow s}$ contains a $(n-s)\times(n-s)$
submatrix with nonzero permanent.
\end{definition}

\begin{lemma}
\label{lem:match-det-0}For any $s\in\{1,\dots,n\}$, and any outcome
of $A^{\uparrow s}$ satisfying $\mathcal{E}(s)$, we have
\[
\Pr[\mathcal{E}(s-1)\,|\,A^{\uparrow s}]\ge1-q^{-s}.
\]
\end{lemma}

\begin{proof}
Throughout this proof, we implicitly condition on the outcome of $A^{\uparrow s}$ (i.e., we treat it as a non-random object). That is to say, we only work with the randomness of the last (i.e., $(n-s+1)$-th) row of $A^{\uparrow (s-1)}$. Denote the entries of this row by $x_1,\dots,x_n$.

Without loss of generality, we may assume that the $(n-s)\times(n-s)$ submatrix
of $A^{\uparrow s}$ formed by the last $n-s$ columns has nonzero permanent. Under this assumption, we will actually only need to use the randomness of $x_1,\dots,x_s$: fix arbitrary outcomes of $x_{s+1},\dots,x_n$, and for the rest of the proof we also implicitly condition on these outcomes.

For $i\in\{1,\dots,s\}$, write $t_{i}$ for the permanent of the $(n-s+1)\times (n-s+1)$
matrix obtained from $A^{\uparrow (s-1)}$ by deleting all of the first
$s$ columns except the $i$th. Then (recalling that we are in a conditional probability space where the only relevant source of randomness is the entries $x_1,\dots,x_s$, and recalling \cref{fact:minor-expansion}), note that $t_{1},\dots,t_{s}$ are all independent, and $t_{i}$ is a bijective 
affine-linear transformation of $x_{i}$. % (with nonzero linear coefficient).
This means that $t_1,\dots,t_s$ are independent random variables that are uniform on $\mb F_q$. The desired result follows, noting that $\mc E(s-1)$ can only fail if $t_1=\dots=t_s=0$.
\end{proof}
Iterating \cref{lem:match-det-0} shows that the permanent is \emph{at
least} as uniform as the determinant, as follows.
\begin{corollary}
\label{lem:match-det}For any $s\in\{1,\dots,n\}$ we have 
\[
\Pr[\mathcal{E}(s-1)]\ge\prod_{i=s}^{n}(1-q^{-i}),
\]
and in particular
\[
\Pr[\on{per}(A)=0]=1-\Pr[\mathcal{E}(0)]\le1-\prod_{i=1}^{n}(1-q^{-i})\le \alpha_{q}.
\]
\end{corollary}

\begin{proof}
We have
\[
\Pr[\mathcal{E}(s-1)]\ge \prod_{i=s}^{n}\Pr[\mathcal{E}(i-1)\,|\,\mathcal{E}(i)]\ge\prod_{i=s}^{n}(1-q^{-i}),
\]
using \cref{lem:match-det-0} and the fact that $\Pr[\mathcal{E}(n)]=1$
(the empty matrix has permanent 1).
\end{proof}

    % Fix a finite field $\mb F_q$, and let $\vec x=(x_1,\dots,x_k)\in \mb F_q^{k}$ be a vector of independent random variables that are uniform on $\mb F_q$. Consider any (non-random) matrix $H\in \mb F_q^{k\times m}$ and vector $\vec h\in \mb F_q^k$. Then $\vec h+H\vec x$ is uniform on some affine subspace of $\mb F_q^n$ with dimension $\on{rank}(H)$, and in particular, for every $\vec z\in \mb F_q$, we have
    % \[\Pr[\vec h+H\vec x=\vec z]\le \frac1{q^{\on{rank(H)}}}.\]

We will also need the following elementary
lemma about $3\times3$ matrices.
\begin{lemma}
\label{lem:matrix}Fix a finite field $\mathbb{F}_{q}$ of odd characteristic,
and let $B\in\mathbb{F}_{q}^{3\times3}$ be a symmetric $3\times3$
matrix whose diagonal entries are all zero and whose off-diagonal
entries are all nonzero. Then $\on{rank}(B)=3$.
\end{lemma}
\begin{proof}
Write $B=(b_{i,j})$, and observe that $\det(B)=2b_{1,2}b_{1,3}b_{2,3}\ne0$.
\end{proof}
The last ingredient we will need is Mantel's theorem from extremal graph theory (see e.g.\ \cite[Theorem~4.6]{Juk01}).
\begin{theorem}\label{thm:mantel}
    If a graph on $n$ vertices has more than $n^2/4$ edges, then it contains a triangle.
\end{theorem}

Now we are ready to prove \cref{thm:main}.
\begin{proof}[Proof of \cref{thm:main}]
First, the desired \emph{lower} bound in \cref{eq:trivial-lower-bound} is near-trivial. Indeed, by \cref{fact:minor-expansion}, the permanent of $A$ is a linear combination of entries of $A$ in the last row, where the coefficients are the permanents of the $(n-1)\times (n-1)$ submatrices of $A^{\uparrow 1}$. Given an outcome of $A^{\uparrow 1}$ satisfying $\mc E(1)$, at least one of the coefficients of this linear combination is nonzero. The entries of the last row of $A$ are independent random variables uniform in $\mb F_q$, so this linear combination is itself uniform in $\mb F_q$. Thus, we obtain (noting that we always have $\on{per}(A)=0$ if the event $\mc E(1)$ does not hold) % writing $\mc E(1)^{\mr c}$ for the complement of the event $\mc E(1)$,
\[
\Pr[\on{per}(A)=0]=(1/q)\Pr[\mathcal{E}(1)]+\Pr[\mathcal{E}(1)^{\mathrm c}]\ge1/q.
\]

For the rest of the proof we focus on upper-bounding $\limsup_{n\to\infty}\Pr[\on{per}(A)=0]$ to show first \cref{eq:separation-all-p} and then \cref{eq:asymptotic-p}. For this, we will use \cref{lem:match-det-0}, but we will treat the last two rows slightly differently. To this end, we define a symmetric matrix $H=(h_{i,j})\in \mb F_q^{n\times n}$ that encodes the permanents of $(n-2)\times (n-2)$ submatrices in $A^{\uparrow2}$: for $i\ne j$, let $h_{i,j}$ be the permanent
of the submatrix of $A^{\uparrow2}$ obtained by deleting columns $i$ and $j$, and for each $i$ let $h_{i,i}=0$. Let
$\vec{x}\in\mathbb{F}_{q}^{n}$ be the last row of $A^{\uparrow 1}$ (i.e., the second-last row of $A$). Then, via \cref{fact:minor-expansion}, it is not hard to see that the entries of $H\vec{x}$ are the permanents of the $(n-1)\times(n-1)$ submatrices in $A^{\uparrow 1}$. Thus, $\mathcal{E}(1)$ holds if and only if $H\vec{x}=\vec{0}$.

If we fix an outcome of $A^{\uparrow 2}$ such that $\on{rank}(H)\ge 3$, then, conditionally on this outcome of $A^{\uparrow 2}$, the random vector $H\vec x$ is uniform on some linear subspace of $\mb F_q^n$ with dimension at least 3. So,
\begin{equation}
\Pr[\mathcal{E}(1)\,|\,\on{rank}(H)\ge 3]=\Pr[H\vec{x}=\vec{0}\,|\,\on{rank}(H)\ge 3]\ge1-q^{-3}.\label{eq:better-with-rank}  
\end{equation}
Note that \cref{eq:better-with-rank} gives a stronger bound than the $s=2$ case of \cref{lem:match-det-0}. The idea for the rest of the proof is to show that $\on{rank}(H)\ge 3$ with non-negligible probability, so we can use \cref{eq:better-with-rank} to improve on the proof of \cref{lem:match-det}.%So, we need to study the probability that $\on{rank}(H)\ge 3$; this will involve \cref{lem:matrix}.

Consider the graph $G$ on the vertex set $\{1,\dots,n\}$ where
$ij$ is an edge if $h_{i,j}\ne0$. The probability that a given
pair $ij$ does not form an edge of $G$ is at most $\alpha_{q}$,
by \cref{lem:match-det} (applied to an $(n-2)\times(n-2)$ random matrix). A direct calculation shows that $\alpha_q<0.45$ (using $q\ge 3$). Thus, by Markov's inequality, the probability
that our graph $G$ has at least $\binom{n}{2}-n^2/4$
non-edges is at most
\[\frac{\alpha_{q}\cdot \binom{n}{2}}{\binom{n}{2}-n^2/4}=2\alpha_{q}+o(1)\le 0.9\]
(assuming, as we may, that $n$ is sufficiently large). If this does not occur,
then by Mantel's theorem (\cref{thm:mantel}) $G$ contains a triangle. This triangle corresponds
to a $3\times3$ principal submatrix of $H$ whose diagonal entries
are all zero and whose off-diagonal entries are all nonzero, so $\on{rank}(H)\ge3$
by \cref{lem:matrix}. In summary, letting $\mathcal{F}$ be the event
that $\on{rank}(H)\ge3$, we have proved that
\[
\Pr[\mathcal{F}]\ge1-2\alpha_q -o(1)\ge 0.1.
\]

Now, we have
\begin{align*}
\Pr[\mathcal{E}(1)] & \ge\Pr[\mathcal{E}(1)\,|\,\mathcal{F}]\Pr[\mathcal{F}]+\Pr[\mathcal{E}(1)\,|\,\mathcal{E}(2)\setminus\mathcal{F}]\Pr[\mathcal{E}(2)\setminus\mathcal{F}]\\
 & \ge(1-q^{-3})\Pr[\mathcal{F}]+(1-q^{-2})(\Pr[\mathcal{E}(2)]-\Pr[\mathcal{F}])\\
 & \ge(1-q^{-2})\Pr[\mathcal{E}(2)]+\Pr[\mathcal{F}](q^{-2}-q^{-3})\\
 & \ge\prod_{i=2}^{n}(1-q^{-i})+\frac{q^{-2}}{20}.
\end{align*}
In the second line we used \cref{eq:better-with-rank} to lower-bound $\Pr[\mc E(1)\,|\,\mc F]$ and we used \cref{lem:match-det-0} to lower-bound
$\Pr[\mathcal{E}(1)\,|\,\mathcal{E}(2)\setminus\mathcal{F}]$. In the last line we used \cref{lem:match-det} to lower-bound $\Pr[\mathcal{E}(2)]$.

Using \cref{lem:match-det-0}, it follows that
\[
\Pr[\on{per}(A)\ne0]=\Pr[\mathcal{E}(0)]\ge\Pr[\mathcal{E}(0)\,|\,\mathcal{E}(1)]\Pr[\mathcal{E}(1)]\ge (1-q^{-1})\Pr[\mc E(1)]\ge\prod_{i=1}^{n}(1-q^{-i})+\frac{q^{-2}}{50}\ge 1-\alpha_q+\frac{q^{-2}}{50}
\]
and hence $\limsup_{n\to \infty}\Pr[\on{per}(A)= 0]\le \alpha_q-q^{-2}/50<\alpha_q$, proving \cref{eq:separation-all-p}.

For \cref{eq:asymptotic-p} we use a similar approach. First, note that the probability that a given pair $ij$ does not form an edge of $G$ can also be bounded via \cref{lem:match-det} by $1-\prod_{i=1}^{n-2}(1-q^{-i})\le \sum_{i=1}^{n-2}q^{-i}\le 2q^{-1}$. Thus, again applying Markov's inequality as before, the probability that $G$ has at least $\binom{n}{2}-n^2/4$
non-edges is at most
\[\frac{2q^{-1}\cdot \binom{n}{2}}{\binom{n}{2}-n^2/4}=2q^{-1}\cdot \frac{2(n-1)}{n-2}\le 8q^{-1},\]
recalling the assumption $n\ge 3$. Thus, we obtain $\Pr[\mathcal{F}]\ge 1-8q^{-1}$.

From \cref{lem:match-det} we furthermore obtain $\Pr[\mathcal{E}(2)]\ge \prod_{i=3}^{n}(1-q^{-i})\ge 1-\sum_{i=3}^{n}q^{-i}\ge 1-2q^{-3}$. Then, the same argument as above shows that
\[\Pr[\mathcal{E}(1)]\ge (1-q^{-2})\Pr[\mathcal{E}(2)]+\Pr[\mathcal{F}](q^{-2}-q^{-3})\ge (1-q^{-2})(1-2q^{-3})+(1-8q^{-1})(q^{-2}-q^{-3})\ge 1-11q^{-3}\]
and consequently
\[\Pr[\on{per}(A)\ne0]\ge (1-q^{-1})\Pr[\mc E(1)]\ge (1-q^{-1})(1-11q^{-3})\ge 1-q^{-1}-11q^{-3}.\qedhere\]
\end{proof}

\begin{remark}
    Here we obtained a separation between the permanent and determinant by considering the last two rows in a special way. One can improve our bounds further by considering the last $s$ rows in a special way, for $s>2$. In particular, with a finite amount of casework for a given value of $s$, it seems that one can improve the error term $O(q^{-3})$ to $O(q^{-s-1})$. (We worked out the details of this for $s=3$, and intend to include them in a future companion note to this paper, but the casework seems to become exponentially more complicated for larger and larger $s$.)% For any individual $s$, there seems to be a finite amount of casework involved, but things seem to become exponentially more complicated for larger and larger $s$. (We worked out the details for $s=3$, and intend to include this in a companion note.)
\end{remark}

\section{Proof for general distributions}\label{sec:general}
Now, we turn to the proof of \cref{thm:general}. Throughout this section, we fix a prime $p\ge 3$ and a probability distribution $\mu$ over $\mb F_p$ supported on more than one value, and we let $A\in\mathbb{F}_{p}^{n\times n}$ be a random matrix with independent $\mu$-distributed entries.

To prove \cref{thm:general}, we need to adapt the proof of \cref{thm:main} to consider \emph{many} submatrices with nonzero permanent, simultaneously. We will consider submatrices satisfying a certain ``complement-disjointness'' property, as follows. % We start with a lemma which tells us that for large $s$, there are very likely to be many $(n-s)\times (n-s)$ submatrices with nonzero permanent, in the first $s$ rows, satisfying a certain disjointness property. To state this lemma, we first introduce some notation.

\begin{definition}\label{def:per-submatrix}
For any $0\le s\le n$, let $A^{\uparrow s}$ be the submatrix of $A$ obtained by deleting the last $s$ rows. For a set $I\subseteq \{1,\dots,n\}$ of size $s$, let $\on{per}(A;I)$ be the permanent of the $(n-s)\times (n-s)$ submatrix of $A^{\uparrow s}$ obtained by deleting the columns indexed by $I$. For a positive integer $\ell$, let $\mc E(s,\ell)$ be the event that there are $\ell$ disjoint sets $I_1,\dots,I_\ell$, each of size $s$, such that $\on{per}(A;I_i)\ne 0$ for each $i=1,\dots,\ell$.
\end{definition}

In the proof of \cref{thm:main} (specifically, in \cref{lem:match-det-0}), we conditioned on an outcome of $A^{\uparrow s}$ which has a submatrix with nonzero permanent, and studied the probability that $A^{\uparrow {(s-1)}}$ has a larger submatrix with nonzero permanent. \cref{lem:match-det-0} was stated only for random matrices with uniform entries; it would be possible to consider general entry distributions, but this would change the probability bound. The following lemma considers a similar situation where we condition on $A^{\uparrow s}$ having \emph{many} submatrices with nonzero permanent, and we would like $A^{\uparrow  (s-1)}$ to have \emph{many} larger submatrices with nonzero permanent. Unfortunately, the two meanings of ``many'' here are rather different from each other, but in this more ``robust'' situation we are able to obtain probability bounds which are essentially independent of the entry distribution (i.e., we obtain approximately the same $1-p^{-s}$ bound as in \cref{lem:match-det-0}, regardless of the entry distribution).

\begin{lemma}\label{lem:one-step-general}
    For any $s,\ell\in \mb N$ and $\varepsilon>0$, there is $L=L_\mu(s,\ell,\varepsilon)\in \mb N$ such that the following holds for sufficiently large $n$. For any outcome of $A^{\uparrow s}$ satisfying $\mc E(s,L)$, we have
    \[\Pr[\mc E(s-1,\ell)\,|\,A^{\uparrow s}]\ge 1-p^{-s}-\varepsilon.\]
\end{lemma}
We defer the proof of \cref{lem:one-step-general} to \cref{sec:approx-uniformity}. It is proved using an anticoncentration inequality for ``robustly high-rank'' linear maps (which is stated as \cref{lem:robust-high-rank-anticoncentration} later in this section, and may be of more general interest).

We would like to use \cref{lem:one-step-general} recursively, to deduce a version of \cref{lem:match-det} for general distributions. Unfortunately, we can only iterate \cref{lem:one-step-general} a very small number of times, as each iteration requires more and more submatrices with nonzero permanent\footnote{We wrote \cref{lem:one-step-general} in a non-quantitative way, so as written it can only be applied $O(1)$ times. With more care, it may be possible to apply it about $\log n$ times, but additional ideas are certainly required.}. So, we also need the following lemma, which serves as a starting point for applying \cref{lem:one-step-general}.

\begin{lemma}\label{lem:tao-vu}
For any $\varepsilon>0$, there is $S=S_\mu(\varepsilon)\in \mb N$  such that \[\Pr[\mc E(S,\lceil n^{0.99}\rceil)]\ge 1-\varepsilon\]for $n$ sufficiently large (in terms of $\varepsilon$ and $\mu$).%if we consider any $s=s_n$ tending to infinity with $n$, such that $s\le n^{o(1)}$, then there is $\ell=n^{1-o(1)}$ such that the following holds. With probability $1-o(1)$, there are disjoint subsets $I_1,\dots,I_\ell\subseteq\{1,\dots,n\}$ of size $s$, such that $\on{per}(A[-I_i])\ne 0$ for each $i$.
\end{lemma}
We defer the proof of \cref{lem:tao-vu} to \cref{sec:growth}. The proof is a straightforward adaptation of ideas of Tao and Vu~\cite{TV09}.

We can now prove a version of \cref{lem:match-det} for  general distributions.

\begin{corollary}
\label{lem:match-det-general}For any $s,\ell\in \mb N$ and $\varepsilon>0$, if $n$ is sufficiently large in terms of $s$, $\ell$, $\varepsilon$ and $\mu$, then
\[
\Pr[\mathcal{E}(s-1,\ell)]\ge\prod_{i=s}^{\infty}(1-p^{-i})-\varepsilon,
\]
and in particular 
\[
\Pr[\on{per}(A)=0]=1-\Pr[\mathcal{E}(0,1)]\le\alpha_{p}+\varepsilon.
\]
\end{corollary}
\begin{proof}
We use the quantities $S_\mu(\cdot)$ and $L_\mu(\cdot,\cdot,\cdot)$ in the statements of \cref{lem:tao-vu,lem:one-step-general}.

    Let $S=S_\mu(\varepsilon/2)$. Let $\ell_{s-1}=\ell$ and for $i\in \{s,\dots,S\}$ let $\ell_i=L_\mu(i,\ell_{i-1},\varepsilon/(2S))$. Then, by \cref{lem:tao-vu} we have $\Pr[\mc E(S,\ell_S)]\ge \Pr[\mc E(S,\lceil n^{0.99}\rceil)]\ge 1-\varepsilon/2$ for sufficiently large $n$, and hence
    \[
\Pr[\mathcal{E}(s-1),\ell]\ge \Pr[\mc E(S,\ell_S)]\prod_{i=s}^{S}\Pr[\mathcal{E}(i-1,\ell_{i-1})\,|\,\mathcal{E}(i,\ell_{i})]\ge\left(1-\frac \varepsilon 2\right)\prod_{i=s}^{S}\left(1-p^{-i}-\frac{\varepsilon}{2S}\right)\ge \prod_{i=s}^{\infty}(1-p^{-i})-\varepsilon,
\]
using \cref{lem:one-step-general}.% and the definitions of $S_\mu(\cdot)$ and $L_\mu(\cdot,\cdot,\cdot)$.
\end{proof}

Now, just as in the proof of \cref{thm:main}, we need to improve \cref{lem:match-det-general} by some special considerations for the last two rows. In order to apply the same proof strategy, we need the following anticoncentration inequality for ``robustly high rank'' linear maps, which may be of general interest. To state this lemma we need a bit more notation.

\begin{definition}\label{def:matrix-indexing}
    For a matrix $M\in \mb F^{m\times n}$ over a field $\mb F$, a set of rows $I\subseteq \{1,\dots,m\}$ and a set of columns $J\subseteq \{1,\dots,n\}$, let $M[I\times J]$ be the $|I|\times |J|$ submatrix obtained by deleting all rows not indexed by $I$ and deleting all columns not indexed by $J$.
\end{definition}

\begin{lemma}\label{lem:robust-high-rank-anticoncentration}
For any $r\in \mb N$ and $\varepsilon>0$, there is $K=K_\mu(r,\varepsilon)\in \mb N$ such that the following holds. Consider a matrix $M\in \mb F_p^{m\times n}$ and disjoint subsets $I_1,\dots,I_{K}\subseteq\{1,\dots,n\}$ with $\on{rank} M[\{1,\dots,m\}\times I_i]\ge r$ for each $i=1,\dots,K$. Let $\vec x\in \mb F_p^n$ be a random vector 
with independent $\mu$-distributed entries. Then
\[\Pr[ M\vec{x}=\vec 0]\le \frac 1{p^r}+\varepsilon.\]
\end{lemma}
We defer the proof of \cref{lem:robust-high-rank-anticoncentration} to \cref{sec:approx-uniformity}. It is proved via a case distinction: either $M$ has high rank, in which case we can easily obtain good anticoncentration bounds for $M\vec x$ (see \cref{lem:simple-bound-linear-algebra}), or we can approximate $M\vec x$ in terms of $r$ independent random variables that are uniform over $\mb F_p$.

The second case in the proof of \cref{lem:robust-high-rank-anticoncentration} requires the following lemma (which we will also need for our deduction of \cref{thm:general}): sums of many independent random variables are approximately uniform. Let us say that a random variable $X\in \mb F_p$ is \emph{$\varepsilon$-almost-uniform} if $\big|\!\Pr[X=z]-1/p\big|\le \varepsilon$ for all $z\in \mb F_p$.

\begin{lemma}\label{lem:approx-unif}
    For any $\varepsilon>0$, there is $Q=Q_\mu(\varepsilon)\in \mb N$ such that the following holds. Consider $\vec h\in \mb F_p^n$, such that least $Q$ entries of $\vec h$ are nonzero, and consider a random vector $\vec x\in \mb F_p^n$
with independent $\mu$-distributed entries. Then $\vec h\cdot \vec x$ is $\varepsilon$-almost-uniform.
\end{lemma}
\begin{proof}
Let $Q$ be very large in terms of $\varepsilon$ and $\mu$, and let $k=\lfloor Q/(p-1)\rfloor$. Without loss of generality we can assume that the first $k$ entries $h_1,\dots,h_{k}$ of $\vec h$ are equal to the same nonzero value $h\in \mb F_p$. Writing $\vec x=(x_1,\dots,x_n)$, the sequence of partial sums
\[x_1,\;x_1+x_2,\;\dots,\; x_1+\dots+x_{k}\] can be interpreted as a Markov chain, and some straightforward analysis (see e.g.\ \cite[Lemma~5.1]{KK01} and the remark thereafter) shows that $x_1+\dots+x_{k}$ is $\varepsilon$-almost-uniform, provided $k$ is sufficiently large in terms of $\varepsilon$ and $\mu$. 
Now, conditioning on any fixed outcomes for $x_{k+1},\dots,x_n$, for any $z\in \mb F_p$, we have 
\[\Pr\!\big[\vec h\cdot \vec x=z\,\big|\,x_{k+1},\dots,x_n\big]=\Pr\!\big[x_1+\dots+x_k=h^{-1}(z-h_{k+1}x_{k+1}-\dots-h_nx_n)\,\big|\,x_{k+1},\dots,x_n\big]\in \Bigg[\frac{1}{p}-\eps,\frac{1}{p}+\eps\Bigg].\]
Finally, averaging over all outcomes of $x_{k+1},\dots,x_n$ yields $\Pr[\vec h\cdot \vec x=z]\in [1/p-\eps,1/p+\eps]$ for all $z\in \mb F_p$, as desired.
\end{proof}

In the proof of \cref{thm:general}, we cannot apply \cref{lem:robust-high-rank-anticoncentration} directly. Instead, we need the following corollary.

\begin{corollary}\label{coro:robust-high-rank-hamming-anticoncentration}
    For any $r,\ell\in \mb N$ and $\varepsilon>0$, there is $D=D_\mu(r,\ell,\varepsilon)\in \mb N$ such that the following holds. Consider a matrix $M\in \mb F_p^{n\times n}$ and disjoint subsets $I_1,\dots,I_{D}\subseteq\{1,\dots,n\}$ of size $r$ with $\on{rank} M[I_i\times I_i]=r$ for each $i=1,\dots,D$. Let $\vec x\in \mb F_p^n$ be a random vector 
with independent $\mu$-distributed entries. Then
\[\Pr[ M\vec{x} \text{ has at most }\ell\text{ nonzero entries }]\le \frac 2{p^r}+\varepsilon.\]
\end{corollary}

\begin{proof}
    Let $K=K_\mu(r,\varepsilon/2)$, where $K_\mu(\cdot,\cdot)$ is as in \cref{lem:robust-high-rank-anticoncentration}, and let $D=2\ell K$. Consider a matrix $M\in \mb F_p^{n\times n}$ and subsets $I_1,\dots,I_{D}\subseteq\{1,\dots,n\}$ as in the assumptions of the corollary. Now, for $j=1,\dots,2\ell$, let $M_j=M[(I_{(j-1)K+1}\cup \dots\cup I_{jK})\times \{1,\dots,n\}]$ be the $(rK)\times n$ submatrix of $M$ consisting of the rows with indices in $I_{(j-1)K+1}\cup \dots\cup I_{jK}$. Since $\on{rank} M_j[I_{(j-1)K+i}\times I_{(j-1)K+i}]=\on{rank} M[I_{(j-1)K+i}\times I_{(j-1)K+i}]=r$ for all $i=1,\dots,K$, the matrix $M_j$ together with the sets $I_{(j-1)K+1},\dots,I_{jK}$ satisfies the assumption of \cref{lem:robust-high-rank-anticoncentration}. Hence we obtain $\Pr[ M_j\vec{x}=\vec{0}]\le p^{-r}+\varepsilon/2$ for all $j=1,\dots,2\ell$.

    Now, if $M\vec{x}$ has at most $\ell$ nonzero entries, then we must have  $M_j\vec{x}=\vec{0}$ for at least $2\ell-\ell=\ell$ indices $j\in \{1,\dots,2\ell\}$. Since the expected number of indices $j\in \{1,\dots,2\ell\}$ with $M_j\vec{x}=\vec{0}$ is at most $(p^{-r}+\varepsilon/2)\cdot 2\ell=(2p^{-r}+\varepsilon)\ell$, by Markov's inequality we can conclude
    \[\Pr[ M\vec{x} \text{ has at most }\ell\text{ nonzero entries }]\le \Pr[ M_j\vec{x}=\vec{0} \text{ for at least }\ell\text{ different } j]\le \frac{(2p^{-r}+\varepsilon)\ell}{\ell}=\frac 2{p^r}+\varepsilon.\qedhere\]
\end{proof}

We now show how to deduce \cref{thm:general} from \cref{lem:one-step-general,lem:match-det-general,coro:robust-high-rank-hamming-anticoncentration,lem:approx-unif}. The structure of the proof is very similar to the proof of \cref{thm:main}, though the notation is a bit more complicated and we need to replace certain trivial estimates with the more sophisticated lemmas collected in this section so far.

\begin{proof}[Proof of \cref{thm:general}]
Consider any $\varepsilon>0$ and any $\ell\in \mb N$. We will prove that for $n$ sufficiently large in terms of $\eps$, $\ell$ and $\mu$, we have
\begin{equation}\Pr[\mc E(1,\ell)]\ge \prod_{i=2}^n(1-p^{-i})+(1-2\alpha_p)(p^{-2}-2p^{-3})-\varepsilon.\label{eq:general-goal}\end{equation}
To see that this suffices, consider $\ell=Q_\mu(\varepsilon)$ (where $Q_\mu(\cdot)$ is as in \cref{lem:approx-unif}). Conditioning on an outcome of $A^{\uparrow1}$ such that $\mc E(1,\ell)$ holds, by \cref{fact:minor-expansion} the permanent $\on{per}(A)$ can be interpreted as a linear combination of the entries in the last row of $A$, where at least $\ell$ coefficients are nonzero. Therefore, \cref{lem:approx-unif} yields
\[\Pr[\on{per}(A)\ne z]\ge \Pr[\on{per}(A)\ne z\,|\,\mc E(1,\ell)]\cdot \Pr[\mc E(1,\ell)]\ge \left(1-\frac1p-\varepsilon\right)\Pr[\mc E(1,\ell)]\ge (1-p^{-1})\Pr[\mc E(1,\ell)]-\varepsilon\]
and, together with \cref{eq:general-goal}, for sufficiently large $n$ we obtain
\[\Pr[\on{per}(A)\ne z]\ge (1-p^{-1})\prod_{i=2}^n(1-p^{-i})+(1-p^{-1})(1-2\alpha_p)(p^{-2}-2p^{-3})-2\varepsilon\ge 1-\alpha_p+(1-2\alpha_p)p^{-2}/5-2\eps.\]
Taking $\varepsilon=(1-2\alpha_p)p^{-2}/20$, this implies \cref{eq:separation-general} (with $\delta_p=(1-2\alpha_p)p^{-2}/10$). Furthermore, taking $\varepsilon=p^{-3}$, \cref{lem:approx-unif} yields
\[\left|\Pr[\on{per}(A)=z]-\frac 1p\right|\le \left|\Pr[\on{per}(A)=z\,|\,\mc E(1,\ell)]-\frac 1p\right|+\Pr[\mc E(1,\ell)^{\mathrm c}]\le \eps+\Pr[\mc E(1,\ell)^{\mathrm c}]=\Pr[\mc E(1,\ell)^{\mathrm c}]+p^{-3}\]
for any $z\in \mb F_p$. Using \cref{eq:general-goal} we can compute (again for sufficiently large $n$)
\begin{align*}
\Pr[\mc E(1,\ell)^{\mathrm c}]=1-\Pr[\mc E(1,\ell)]&\le 1-\prod_{i=2}^n(1-p^{-i})-(1-2\alpha_p)(p^{-2}-2p^{-3})+\varepsilon\\
&\le \sum_{i=2}^n p^{-i}-(p^{-2}-2p^{-3})+2\alpha_pp^{-2}+2p^{-3}\\
&\le 2p^{-3}+2p^{-3}+2\alpha_pp^{-2}+2p^{-3}\le 10p^{-3},
\end{align*}
using that $\alpha_p=1-\prod_{i=1}^\infty (1-p^{-i})\le \sum_{i=1}^\infty p^{-i}\le 2p^{-1}$. Thus, we obtain
\[\left|\Pr[\on{per}(A)=z]-\frac 1p\right|\le 10p^{-3}+p^{-3}=11p^{-3}\]
for any $z\in \mb F_p$, proving \cref{eq:asymptotic-general} with $C=11$.

So, we just need to prove \cref{eq:general-goal}. Define the matrix $H\in \mb F_p^{n\times n}$ and the graph $G$ as in the proof of \cref{thm:main} (encoding the permanents of $(n-2)\times (n-2)$ submatrices in $A^{\uparrow2}$). Let $D_\mu(\cdot,\cdot,\cdot)$ be as in \cref{coro:robust-high-rank-hamming-anticoncentration}, and let $D=D_\mu(3,\ell,\varepsilon/4)$. By Markov's inequality and \cref{lem:match-det-general}, the probability that our graph $G$ has at least $\binom n 2-n^2/4-3Dn$ non-edges is at most $2\alpha_p+o(1)$. If this occurs, then we can repeatedly apply Mantel's theorem (\cref{thm:mantel}) to find $D$ vertex-disjoint triangles. (After each time we apply Mantel's theorem to find a triangle, we delete all edges incident to the vertices of that triangle before the next application of Mantel's theorem.)

Just as in the proof of \cref{thm:main}, via \cref{lem:matrix}, a triangle in $G$ yields a size-$3$ set $I\subseteq\{1,\dots,n\}$ such that $\on{rank}(H[I\times I])=3$. Write $\mc F^*$ for the event that there are $D$ disjoint size-$3$ sets $I_1,\dots,I_{D}\subseteq\{1,\dots,n\}$ such that $\on{rank}(H[I_i\times I_i])=3$ for each $i=1,\dots,D$; we have just proved that 
\[\Pr[\mc F^*]\ge 1-2\alpha_p+o(1).\]

For any outcome of $A^{\uparrow2}$ such that $\mc F^*$ holds, we can apply \cref{coro:robust-high-rank-hamming-anticoncentration}  to the matrix $H\in \mb F_p^{n\times n}$ and find that  with probability at least $1-2p^{-3}-\eps/4$ the vector $H\vec{x}$ has at least $\ell$ nonzero entries, where $\vec{x}\in F_p^{n}$ denotes the random vector consisting of the entries in the second-last row of $A$. Now, the entries of $H\vec{x}$ are precisely the permanents of the $(n-1)\times (n-1)$ submatrices of $A^{\uparrow1}$, and having at least $\ell$ such nonzero entries means that the event $\mc E(1,\ell)$ holds. Thus, we obtain
\[\Pr[\mc E(1,\ell)\,|\,\mc F^*]\ge 1-2p^{-3}-\eps/4.\]

Now, let $L_\mu(\cdot,\cdot,\cdot)$ be as in \cref{lem:one-step-general} and write $\mc E^*(2)=\mc E(2,L_\mu(2,\ell,\varepsilon/4))$. Then by \cref{lem:match-det-general} we have
\[\Pr[\mc E^*(2)]\ge \prod_{i=3}^\infty(1-p^{-i})-\varepsilon/4\]
if $n$ is sufficiently large.

Combining everything, and applying \cref{lem:one-step-general}, we can conclude
\begin{align*}
    \Pr[\mc E(1,\ell)]&\ge \Pr[\mc E(1,\ell)\,|\,\mc F^*]\Pr[\mc F^*]+\Pr[\mc E(1,\ell)\,|\,\mc E^*(2)\setminus \mc F^*]\Pr[\mc E^*(2)\setminus \mc F^*]\\
   &\ge (1-2p^{-3}-\varepsilon/4)\Pr[\mc F^*]+(1-p^{-2}-\varepsilon/4)\Pr[\mc E^*(2)\setminus \mc F^*]\\
   &\ge (1-2p^{-3})\Pr[\mc F^*]+(1-p^{-2})\Pr[\mc E^*(2)\setminus \mc F^*]-\varepsilon/2\\
   &\ge (1-2p^{-3})\Pr[\mc F^*]+(1-p^{-2})\left(\prod_{i=3}^\infty(1-p^{-i})-\varepsilon/4-\Pr[\mc F^*]\right)-\varepsilon/2\\
   &\ge\prod_{i=2}^\infty(1-p^{-i})+(p^{-2}-2p^{-3})\Pr[\mc F^*]-3\varepsilon/4\\
    &\ge\prod_{i=2}^\infty(1-p^{-i})+(p^{-2}-2p^{-3})(1-2\alpha_p+o(1))-3\varepsilon/4,
    %&\ge (1-p^{-3}-\varepsilon/3)\Pr[\mc F^*]+(1-p^{-2}-\varepsilon/3)\left(\prod_{i=3}^\infty(1-p^{-i})-\varepsilon/3-\Pr[\mc F^*]\right),
\end{align*}
and \cref{eq:general-goal} follows for sufficiently large $n$.
\end{proof}

\subsection{Growing many matrices with nonzero permanents}\label{sec:growth}
In this section we prove \cref{lem:tao-vu}, adapting a strategy of Tao and Vu~\cite{TV09}. Recall the definitions in \cref{def:per-submatrix}. First, the following lemma is suitable to find a \emph{single} submatrix with nonzero permanent.

\begin{lemma}\label{lem:growth}
For any $\varepsilon>0$, there is $T=T_\mu(\varepsilon)\in\mb N$ such that the following
holds. Consider any subset $J\subseteq\{1,\dots,n\}$ of size $|J|=2T$.
Then,
\[
\Pr[\text{there is a}\text{ subset }I\subseteq J\text{ of size }|I|=T\text{ with }\on{per}(A;I)\ne0]\ge 1-\varepsilon.
\]
\end{lemma}

\begin{proof} Let $\rho=\max_{z\in\mb F_{p}}\mu(z)<1$ (recalling that $\mu$
is supported on at least two values), and choose $0<\delta<1$
small enough such that $1-\rho(1+\delta)>\delta$. Furthermore choose $T>2(1+\delta)/(\eps\delta^2)$ to be sufficiently large such that $\rho^{\delta T}/(1-\rho)<\eps/4$.

We consider a random process that attempts to construct a nested sequence
of sets $I_{0}\supseteq\dots\supseteq I_{n-T}$ (starting with $I_{0}=\{1,\dots,n\}$,
and removing an element at each step), in such a way that $\on{per}(A;I_{t})\ne0$
for all $t=0,\dots,n-T$, and $I_{n-T}\subseteq J$. If this process succeeds,
then we can take $I_{n-T}$ as the desired set $I$.

Specifically, the process is defined as follows. Let $T'=2T+\lfloor \delta T\rfloor$. For each step $t=1,\dots,n-T'$, we define $I_t$ as follows:
\begin{itemize}
\item If there is $i\in I_{t-1}\setminus J$ such that $\on{per}(A;I_{t-1}\setminus\{i\})\ne0$, take the minimum such $i$ and
set $I_{t}=I_{t-1}\setminus\{i\}$.
\item Otherwise, terminate the process.
\end{itemize}
For each step $t=n-T'+1,\dots, n-T$, we define $I_t$ as follows:
\begin{itemize}
\item If there is $i\in I_{t-1}\setminus J$ such that $\on{per}(A;I_{t-1}\setminus\{i\})\ne0$, take the minimum such $i$ and
set $I_{t}=I_{t-1}\setminus\{i\}$.
\item Otherwise, if there is any $i\in I_{t-1}$ such that $\on{per}(A;I_{t-1}\setminus\{i\})\ne0$, take the minimum such $i$ and
set $I_{t}=I_{t-1}\setminus\{i\}$.
\item Otherwise, terminate the process.
\end{itemize}
That is to say: in each of the first $n-T'$ steps we make sure to
remove an element outside $J$. In each of the subsequent
steps we still try to remove an element outside $J$ if possible,
but if not we allow ourselves to remove some other element. Note that we always have $|I_t|=n-t$ and $\on{per}(A;I_{t})\ne0$ whenever the set $I_t$ is defined, and also note that $I_t$ only depends on the outcome of $A^{\uparrow (n-t)}$ (i.e., only on the first $t$ rows of the matrix $A$).

It suffices to show that with probability at least $1-\varepsilon/2$
this process does not terminate in any step before defining $I_{n-T}$, and to show
that\footnote{If we terminate in some step (so $I_{n-T}$ is not defined), we say that the event $\{I_{n-T}\not \subseteq J\}$ does \emph{not} occur.} $I_{n-T}\nsubseteq J$ with probability at most $\varepsilon/2$.

First, note that for any outcome
of $A^{\uparrow (n-t)}$ for which $\on{per}(A;I_{t})\ne0$, we have
\[
\Pr[\text{the process terminates at step }t+1\,|\,A^{\uparrow (n-t)}]\le\begin{cases}
\rho^{|I_{t}\setminus J|}\le\rho^{n-t-2T} & \text{if }t<n-T'\\
\rho^{|I_{t}|}=\rho^{n-t}& \text{if }n-T'\le t< n-T
\end{cases}
\]
The reasoning is basically the same as in the proof of \cref{lem:match-det-0}. Indeed,
write $A=(a_{i,j})$, condition on an outcome of $A^{\uparrow (n-t)}$,
and additionally condition on $a_{t+1,i}$ for all $i\notin I_{t}$.
Then, the quantities $\on{per}(A;I_{t}\setminus\{i\})$, for
$i\in I_{t}$, become independent random variables (each only depending on the outcome of $a_{t+1,i}$), which each have
probability at most $\rho$ of taking any particular value (in particular,
of taking the value zero).

By a union bound, it follows that the probability that the process terminates in some step (i.e., the probability that $I_{n-T}$ does not get defined) is at most

\[
\sum_{t=0}^{n-T'-1}\rho^{n-t-2T}+\sum_{t=n-T'}^{n-T-1}\rho^{n-t}=\sum_{j=T'+1-2T}^{n-2T}\rho^{j}+\sum_{j=T+1}^{T'}\rho^{j}\le \frac{\rho^{\lfloor \delta T\rfloor+1}+\rho^{T+1}}{1-\rho}\le \frac{2\rho^{\delta T}}{1-\rho}\le\frac{\varepsilon}{2}.
\]

Now we study the probability that $I_{n-T}\nsubseteq J$. For $t= n-T'+1,\dots,n-T$, say step $t$ is \emph{bad} if $I_{t}\setminus J=I_{t-1}\setminus J\ne \emptyset$ (i.e., if $I_{t-1}\nsubseteq J$, and the element we remove from $I_{t-1}$ in step $t$ lies in $J$). Then for any outcome of $A^{\uparrow t}$ with $I_{t-1}\nsubseteq J$, we have
\[
\Pr[t\text{ is bad}\,|\,A^{\uparrow t}]\le\rho^{|I_{n-t}\setminus J|}\le\rho,
\]
and for any outcome of $A^{\uparrow t}$ with $I_{t-1}\subseteq J$ we have $\Pr[t\text{ is bad}\,|\,A^{\uparrow t}]=0\le \rho$. Thus, the number of bad steps is stochastically dominated by a $\on{Binomial}(T'-T,\rho)$-distributed random variable.

Now, if $I_{n-T}\not\subseteq J$, there must be an element $h\in I_{n-T}\not\subseteq J$. This means that $I_{t}\not\subseteq J$ for all $t= n-T'+1,\dots,n-T$. Furthermore, at most $|I_{n-T'}\setminus J|\le |I_{n-T'}|-|J|=T'-2T=\lfloor \delta T\rfloor$ elements outside $J$ get removed during the $T'-T$ steps $n-T'+1,\dots,n-T$. This means that at least $T'-T-\lfloor \delta T\rfloor=T$ of these steps must be bad.

Thus, writing $B\sim\on{Binomial}(T'-T,\rho)$ (which has expectation $\rho(T'-T)$ and standard deviation at
most $\sqrt{T'-T}\le \sqrt{(1+\delta)T}$), we have
\begin{align*}
\Pr[I_{n-T}\not\subseteq J] & \le\Pr[\text{at least }T\text{ of the steps }n-T'+1,\dots,n-T\text{ are bad}]\\
 & \le\Pr[B\ge T]\\
 & \le\Pr[B\ge\rho(T'-T)+\delta T]\le\frac{(1+\delta)T}{\delta^2 T^2}=\frac{(1+\delta)}{\delta^2 T}\le \frac{\varepsilon}{2},
\end{align*}
noting that $\rho(T'-T)+\delta T\le \rho(1+\delta) T+\delta T\le T$ by definition of $T'$ and of $\delta$, and using Chebyshev's inequality.
\end{proof}

Now we show how to deduce \cref{lem:tao-vu}.

\begin{proof}[Proof of \cref{lem:tao-vu}]
 %Let $\rho$ be as in the statement of \cref{lem:growth}, and l
 Let $T_\mu(\cdot)$ be as in \cref{lem:growth} and let $S=T_\mu(\varepsilon/2)$.
Let $\ell=2\lceil n^{0.99}\rceil$
and fix disjoint subsets $J_{1},\dots,J_{\ell}\subseteq\{1,\dots,n\}$
of size $2S$. Say $J_{i}$ is \emph{good} if there is a
subset $I_i\subseteq J_i$ of size $|I_i|=S$ with $\on{per}(A;I_i)\ne0$, and say that $J_{i}$ is \emph{bad} otherwise. Then each $J_{i}$
is bad with probability at most $\varepsilon/2$ by
\cref{lem:growth}. By Markov's inequality, the probability that more than half
of $J_{1},\dots,J_{\ell}$ are bad is at most $2(\varepsilon/2)=\varepsilon$. If this does not happen, i.e.\ if at least $\ell/2=\lceil n^{0.99}\rceil$ of the sets $J_{1},\dots,J_{\ell}$ are good, then $\mc E(S,\lceil n^{0.99}\rceil)$ occurs.
\end{proof}

\subsection{Approximate uniformity for linear maps}\label{sec:approx-uniformity}
In this section we prove \cref{lem:robust-high-rank-anticoncentration} and use it to deduce \cref{lem:one-step-general}. First, we need the following simple lemma showing anticoncentration bounds for $M\vec x$ in terms of the rank of $M$ and the entry distribution of $\vec x$.

\begin{lemma}\label{lem:simple-bound-linear-algebra}
Let $\rho=\max_{z\in\mb F_{p}}\mu(z)<1$, and let $\vec{x}\in F_{p}^n$
be a vector with independent $\mu$-distributed entries. Then for any matrix $M\in\mb F_{p}^{m\times n}$ and any
$\vec{z}\in\mb F_{p}^{m}$, we have
\[
\Pr[M\vec{x}=\vec{z}]\le\rho^{\on{rank} M}.
\]
\end{lemma}
\begin{proof}
    We can interpret $M\vec{x}=\vec{z}$ as a system of linear equations. Letting $r=\on{rank} M$, this system has exactly $n-r$ free variables. Conditioning on any entries for the entries of $\vec{x}$ corresponding to these free variables, there is at most one choice for each of the remaining $r$ entries such that $M\vec{x}=\vec{z}$. For each of these entries, the probability of attaining the appropriate value is at most $\rho$, so overall $\Pr[M\vec{x}=\vec{z}]$ is at most $\rho^r=\rho^{\on{rank} M}$.
\end{proof}

Now, we are ready to prove \cref{lem:robust-high-rank-anticoncentration}.

\begin{proof}[Proof of \cref{lem:robust-high-rank-anticoncentration}]
We may assume without loss of generality $\eps<1/2$. Let $\rho=\max_{z\in\mb F_{p}}\mu(z)<1$ be the maximum point probability of $\mu$, and let $R\in \mb N$ be such that $\rho^R\le p^{-r}$. Furthermore, let $Q_\mu(\cdot)$ be as in \cref{lem:approx-unif}, let $Q=Q_\mu(\varepsilon/r)$ and let $K=2p^{R}Q$.

First, consider the case that $\on{rank} M> R$. Then by \cref{lem:simple-bound-linear-algebra} we have
\[\Pr[M\vec{x}=\vec{0}]\le\rho^{\on{rank} M}\le \rho^R\le p^{-r}+\eps,\]
as desired.

So let us from now on assume that $\on{rank}(M)\le R$. Then there is a set $I\subseteq\{1,\dots,m\}$ of size $|I|\le R$, such that the rows indexed by $I$ generate the row space of the matrix $M$. In particular, for any $i=1,\dots,K$ the row space of the matrix $M[\{1,\dots,m\}\times I_i]$ is generated by the rows indexed by $I$, and so our assumption $\on{rank} M[\{1,\dots,m\}\times I_i]\ge r$ implies $\on{rank} M[I\times I_i]\ge r$. Thus, for each $i=1,\dots,K$, the submatrix $M[I\times I_i]$ has $r$ linearly independent columns.

Let $S\subseteq\mb F_{p}^{I}$ be the set of ``popular'' vectors that appear at least $K/(2p^{R})=Q$ times as columns of $M[I\times \{1,\dots,n\}]$. All but at most $p^{|I|}\cdot K/(2p^{R})\le p^{R}\cdot K/(2p^{R})=K/2$ columns of $M[I\times \{1,\dots,n\}]$ are popular, so there is an index $i\in\{1,\dots,K\}$ such that all columns of $M[I\times I_i]$ are popular vectors. This means in particular that there are $r$ linearly independent popular vectors $\vec{v}_{1},\dots,\vec{v}_{r}\in S\subseteq \mb F_{p}^{I}$. After reordering the columns of $M$, suppose without loss of generality that the first
$Q$ columns of $M[I\times \{1,\dots,n\}]$ are equal to $\vec{v}_{1}$, the second $Q$ columns are equal to $\vec{v}_{2}$, and so on. Then, whenever we have $M\vec{x}=\vec{0}$, we obtain $M[I\times \{1,\dots,n\}]\vec{x}=\vec{0}$, and hence
\[(x_1+\dots+x_Q)\vec{v}_{1}+\dots+(x_{(r-1)Q+1}+\dots+x_{rQ})\vec{v}_{r}+M[I\times\{rQ+1,\dots,n\}](x_{rQ+1},\dots,x_{n})^{T}=M\vec{x}=\vec{0}.\]
Thus, writing $\vec{z}=-M[I\times\{rQ+1,\dots,n\}](x_{rQ+1},\dots,x_{n})^{T}$ and $y_{j}=x_{(j-1)Q+1}+\dots+x_{jQ}$ for $j=1,\dots,r$, we obtain
\[
\Pr[M\vec{x}=\vec{0}]\le\Pr\!\big[M[I\times \{1,\dots,n\}]\vec{x}=\vec{0}\big]\le \Pr[y_{1}\vec{v}_{1}+\dots+y_{r}\vec{v}_{r}=\vec{z}].
\]
Now, since $\vec{v}_{1},\dots,\vec{v}_{r}\in F_{p}^{I}$ are linearly independent, for any outcome of $\vec{z}$ there is at most  one possible outcome of $y_{1},\dots,y_{r}\in\mb F_{p}$ such that $y_{1}\vec v_{1}+\dots+y_{r}\vec v_{r}=\vec{z}$.
Note that $y_{1},\dots,y_r$ are independent, and by \cref{lem:approx-unif} each $y_i$ is $(\varepsilon/r)$-almost-uniform, so we may conclude that
\[
\Pr[M\vec{x}=\vec{0}]\le\Pr[y_{1}\vec{v}_{1}+\dots+y_{r}\vec{v}_{r}=\vec{z}]\le \left(\frac{1}{p}+\frac{\varepsilon}{r}\right)^{r}\le\frac{1}{p^r}+\varepsilon.\qedhere
\]
\end{proof}

Now we show how to deduce \cref{lem:one-step-general}.
\begin{proof}[Proof of \cref{lem:one-step-general}]
    For large $L$, fix an outcome of $A^{\uparrow s}$ satisfying $\mc E(s,L)$, so there are disjoint size-$s$ sets $I_1,\dots,I_L$, each satisfying $\on{per}(A;I_i)\ne 0$. For each $j\in\{1,\dots,L\}$, let $\mathcal{F}_{j}$
be the event that $\on{per}(A;I_{j}\setminus\{i\})=0$ for all $i\in I_{j}$,
and let $X$ be the number of $j\in \{1,\dots,L\}$ for which $\mathcal{F}_{j}$ occurs.
Our objective is to prove that, for sufficiently large $L$, we have
\[
\Pr[X>L-\ell]\le\frac{1}{p^{s}}+\varepsilon.
\]
We will do this by studying a sufficiently high moment of $X$, using \cref{lem:robust-high-rank-anticoncentration}.

Write $\vec{x}$ for the $(n-s+1)$-th row of $A$. First, note that
(given our outcome of $A^{\uparrow s}$), for each $j\in\{1,\dots,L\}$,
the quantities $\on{per}(A;I_{j}\setminus\{i\})$, for $i\in I_{j}$,
depend linearly on $\vec{x}$. We define an $s\times n$ matrix $M_{j}$
to record this linear dependence (with rows indexed by
$I_{j}$, and with columns indexed by $\{1,\dots,n\}$),
as follows. For $i\in I_{j}$ and $h\notin I_{j}$, write $I_{j}-i+h$
to denote the size-$s$ set $(I_{j}\setminus\{i\})\cup\{h\}$. Then,
for $i\in I_{j}$ and $h\in\{1,\dots,n\}$, define the $(i,h)$-entry
of $M_{j}$ to be
\[
\begin{cases}
\on{per}(A;I_{j}-i+h) & \text{if }h\notin I_{j},\\
\on{per}(A;I_{i}) & \text{if }i=h,\\
0 & \text{if }h\in I_{j}\setminus\{i\}.
\end{cases}
\]
It is not hard to see (using \cref{fact:minor-expansion}) that for any $i\in I_j$, the $i$-entry of $M_{j}\vec{x}$
is precisely $\on{per}(A;I_{j}\setminus\{i\})$. In particular, $\mathcal{F}_{j}$ occurs if and
only if $M_{j}\vec{x}=\vec{0}$. Also note that the $s\times s$ submatrix $M_j[I_j\times I_j]$ is a diagonal matrix with all entries on the diagonal being $\on{per}(A;I_{i})\ne 0$, so we have $\on{rank} M_j[I_j\times I_j]=s$.

Now, let $K_\mu(\cdot,\cdot)$ be as in \cref{lem:robust-high-rank-anticoncentration} and let $K=K_\mu(s,\varepsilon/2)$. For
each subset $J\subseteq\{1,\dots,L\}$ with $|J|=K$, let $M_{J}\in\mb F_{p}^{(Ks)\times n}$
be the matrix obtained by concatenating the matrices $M_{j}$ for
$j\in J$ below each other (i.e.\ the matrix whose rows are the rows of all the matrices $M_{j}$ for
$j\in J$). Then, $M_{J}$ satisfies the conditions of \cref{lem:robust-high-rank-anticoncentration} with $r=s$ (since $\on{rank} M_J[\{1,\dots,Ks\}\times I_j]\ge \on{rank} M_j[I_j\times I_j]=s$ for all $j=1,\dots,K$), and therefore
\[
\Pr\Bigg[\bigcap_{j\in J}\mathcal{F}_{j}\Bigg]=\Pr[M_{J}\vec{x}=\vec{0}]\le\frac{1}{p^{s}}+\frac{\varepsilon}{2}.
\]
This implies that
\[
\mb E\left[\binom{X}{K}\right]\le\left(\frac{1}{p^{s}}+\frac{\varepsilon}{2}\right)\binom{L}{K},
\]
so by Markov's inequality,
\[
\Pr[X>L-\ell]\le\Pr\left[\binom{X}{K}\ge \binom{L-\ell}{K}\right]\le\left(\frac{1}{p^{s}}+\frac{\varepsilon}{2}\right)\binom{L}{K}\binom{L-\ell}{K}^{-1}\le\frac{1}{p^{s}}+\varepsilon
\]
for sufficiently large $L$.
\end{proof}

\bibliographystyle{amsplain_initials_nobysame_nomr}
\bibliography{main.bib}
\end{document}